\documentclass[10pt]{amsart}

\usepackage{setspace, amsmath}
\usepackage{fancyhdr, lastpage, amssymb}
\usepackage{verbatim}
\usepackage[mathscr]{eucal}
\usepackage{float}

\newtheorem{theorem}{Theorem}[section]

\theoremstyle{definition}
\newtheorem{definition}[theorem]{Definition}
\newtheorem{example}[theorem]{Example}

\newtheorem{prop}[theorem]{Proposition}
\newtheorem{corr}[theorem]{Corollary}

\theoremstyle{remark}

\newcommand{\Z}{\mathbb{Z}}

\newcommand{\gameset}[1]{\mathcal{#1}}
\newcommand{\GL}{\gameset{L}}
\newcommand{\GR}{\gameset{R}}

\newcommand{\graph}[1]{\mathscr{#1}}
\newcommand{\path}[1]{\graph{P}_{#1}}
\newcommand{\complete}[1]{\graph{K}_{#1}}

\newcommand{\leftmove}[2]{{#1}~\overrightarrow{_{L}}~{#2}}
\newcommand{\rightmove}[2]{{#1}~\overrightarrow{_{R}}~{#2}}

\newcommand{\game}[2]{\left\{{#1}~|~{#2}\right\}}
\newcommand{\zerogame}{\game{}{}}

\newcommand{\farstareq}[2]{{#1}\sim_{\bigstar}{#2}}

\newcommand{\hidethis}[1]{}
\def\Comment#1{\marginpar{$\bullet$\quad{\tiny #1}}}

\def\qed{\hfill\hbox{${\vcenter{\vbox{
    \hrule height 0.4pt\hbox{\vrule width 0.4pt height 6pt
    \kern5pt\vrule width 0.4pt}\hrule height 0.4pt}}}$}}
 
\numberwithin{equation}{section}

\begin{document}
\thispagestyle{empty}


\title{Partisan Combinatorial Game of Edge and Vertex Removal on Graphs}

\author{Nathan Shank}
\address{Moravian College}
\curraddr{1200 Main Street Bethlehem, PA 18018}
\email{shank@math.moravian.edu, dsvukovich13@gmail.com}

\author{Devon Vukovich}



\date{August 2020}


\begin{abstract}
We consider three variants of a partisan combinatorial game between two players, Left and Right, played on an undirected simple graph. Left is able to delete vertices (and incident edges) while Right is able to delete edges.  
This natural extension of a similar impartial game gives a clear advantage to one player by allowing them the ability to play on a small subgraph which the other player can not.  Our last variant removes this advantage by assuming a move is valid for one player if and only if the other player has a valid move on the same graph.  In this case, we show that the ability to remove a vertex is more advantageous compared to removing edges.  

\end{abstract}

\maketitle


\section{Introduction}

Many games, both partial and impartial, can be played on graphs.  Some games generate graphs of larger size and/or order, for example \textit{Sprouts} \cite{Sprouts}.  While others remove edges and/or vertices to reduce the graph, for example the impartial version of \textit{Hackenbush} \cite{LessonsInPlay}.  In either case, the game ends when a particular end state of the graph is reached.  For edge deletion games a \textit{turn} consists of removing an edge from the graph and possibly additional edges or vertices depending on the structure of the game.  For example, Gallant et. al. \cite{GameEdgeRemoval} considered deleting edges from a graph so long as the graph does not contain an isolated vertex. Similarly in a vertex deletion game, a \textit{turn} consists of removing a vertex which may also result in other vertices or edges being removed.  For example Adams et. al. \cite{SubtractionGame} considered a game where a turn consists of removing a vertex, all incident edges, and all vertices which become isolated. 

Often, the ability for a vertex or edge to be removed is dependent on some properties of the object.  Nowakowski and Ottaway \cite{VertexParity} considered a vertex removal game based on the parity of the vertex while Harding and Ottaway \cite{EdgeParity} considered variants of an edge removal game based on the parity of the incident vertices.  In either case, the underlying principle of the game clearly define what objects may or may not be removed.  

 In this paper we consider three partisan variants of the impartial edge removal game presented in \cite{GameEdgeRemoval}.  Two players, Left and Right, play on a finite graph $G$.  A \textit{turn} for Left consists of removing a vertex and all incident edges while a \textit{turn} for Right consists of removing a single edge.  The variants differ based on the underlying principles which define what objects can be removed.  In section \ref{MixedGame} we will analyze the \textit{Classic Variant} which follows the model introduced by Gallant et. al. \cite{GameEdgeRemoval} which implies the game continues so long as the graph does not contain an isolated vertex.  In section \ref{ForbiddenGame}, the \textit{Forbidden Leaf Variant} will impose an additional restriction not allowing the deletion of leaf vertices.  Finally, in section \ref{MutualGame}, the \textit{Mutual Failures Variant} will impose additional restrictions which will allow Left the ability to delete a vertex from a subgraph if and only if Right has the ability to delete an edge from the same subgraph.  
 
 As in the original result of Gallant et. al. \cite{GameEdgeRemoval}, we will concern ourselves with finding which values of $n$ in a particular graph class can be won by a particular player, rather than the value of the game or the size of the advantage, although these would be interesting additions for future consideration.  


\section{Background and definitions}

We will adopt standard notation from game theory literature.  The theorems and notation introduced in this section can be found in \cite{LessonsInPlay} and \cite{WinningWays}.  

Since we are interested in a game with only two players, Left and Right, we will denote a game position, $G$, as $G=\game{\GL}{\GR}$, where $\GL$ and $\GR$ represent the move sets available to Left and Right respectively.  When Left moves from $G$ to some option $L\in\GL$, we will note her move $\leftmove{G}{L}$. Similarly, $\rightmove{G}{R}$ is a move by Right from $G$ to some $R\in\GR$.  The zero game will be denoted $G = \zerogame$, star will be denoted $\ast = \game{0}{0}$, far star will be denoted $\bigstar$, the game position $\game{0}{\ast}$ is defined as up $(\uparrow)$ and the negative of up, $\game{\ast}{0}$ is defined as down, $(\downarrow)$.  We will write $G\|H$ if $G$ is incomparable with $H$.  If $G>H$ or $G\|H$ we will denote this by $G \rhd H$. 

Since our analysis will be primarily focused on all-small games, we will provide a review of some useful definitions and theorems related to atomic weights.  

\subsection{Atomic Weights}\label{AtomicWeights}

When examining an all-small game, complicated infinitesimal positions are frequently encountered. Many of the analysis techniques used to make sense of a game as a whole are useless when applied to such positions. It then becomes desirable to relate a complicated position to one that is understood. In this way, the dominant features of a position can be equated to the features of a known position provided the error can be accommodated. We use an equivalence relation, far star, to do just this.

\begin{definition}\label{d:BG:FarStarEquiv}

We say a game position $G$ is far star equivalent to a game position $H$ if $G+X+\bigstar$ and $H+X+\bigstar$ have the same outcome for all game positions $X$. We denote this equivalence $\farstareq{G}{H}$.
\end{definition}

The following establishes a useful way to determine far star equivalence.

\begin{theorem}\label{t:BG:UsefulEquiv}

For game positions $G$ and $H$, $\farstareq{G}{H}$ if and only if ${\downarrow}{\bigstar} < G-H < {\uparrow}{\bigstar}$.
\end{theorem}

Theorem \ref{t:BG:UsefulEquiv} provides a scale based upon $\uparrow$ that we can use to measure infinitesimals. Recall that despite being an infinitesimal and consequently less than all positive numbers, $\uparrow>0$. This makes $\uparrow$ a Left win, and similarly $\downarrow$ is a Right win. To be able to draw an equivalence between an infinitesimal position and some integer multiple of $\uparrow$ would allow us to measure how much an infinitesimal position favors a player.

\begin{definition}\label{d:BG:AtomicWeight}
For an integer $n$ and a game position $G$, the atomic weight of $G$ is $n$ if $\farstareq{G}{n\cdot \uparrow}$. We denote this AW$(G)=n$.
\end{definition}

The above definition restricts the atomic weight to an integer value, but there is no reason the atomic weight must be integer valued. We make this restriction as our usage of atomic weights will yield integer values.  In the above definition, $n\cdot\uparrow$ represents the game sum of $|n|$ copies of $\uparrow$ if $n$ is positive and $|n|$ copies of $\downarrow$ if $n$ is negative. Atomic weights are well defined, and because far star equivalence respects addition, atomic weights are also additive. 

\begin{theorem}\label{t:BG:AtomicAdd}

For game positions $G$ and $H$, 
$$\text{AW}(G)+\text{AW}(H)=\text{AW}(G+H).$$
\end{theorem}

Clearly, this result allows for the atomic weight of a game sum to be easily computed by summing the atomic weights of each summand. We now present the result that makes atomic weights the most effective tool for solving infinitesimal positions.

\begin{theorem}\label{t:BG:TwoAhead}
\doublespacing
For a game position $G$,

\begin{align*}
\text{if AW}(G)&\geq 2 \text{, then }G>0, \\[-10pt]
\text{if AW}(G)&\geq 1 \text{, then }G \rhd 0, \\[-10pt]
\text{if AW}(G)&\leq -1 \text{, then }G \lhd 0, \\[-10pt]
\text{and if AW}(G)&\leq -2 \text{, then }G<0.
\end{align*}

\end{theorem}

This theorem is known as the \textit{Two Ahead Rule}. If the atomic weight is at least two in favor of a player, that player then has sufficient advantage in the position to win regardless which player moves first. While it is possible to compute the atomic weight of a position $G$ using only the definition of atomic weight, the following theorem, known as the \textit{Atomic Weight Calculus}, provides a method to compute the atomic weight using only the atomic weights of the positions in $G$. 

\begin{theorem}\label{t:BG:AtomicCalc}
\doublespacing
For a game position $G=\game{L_1,L_2,L_3,\cdots}{R_1,R_2,R_3,\ldots}$ if $L_i$ and $R_i$ have atomic weights for all $i$, then 
\begin{align*}
\text{AW}(G)=\game{\text{AW}(L_1)-2,\text{AW}(L_2)-2,\ldots}{\text{AW}(R_1)+2,\text{AW}(R_2)+2,\ldots}
\end{align*}
unless this is an integer. If this is is an integer, let $x$ be the least integer such that,
\begin{align*}
\text{AW}(L_1)-2\lhd x,\text{AW}(L_2)-2\lhd x,\ldots
\end{align*}
and let $y$ be the greatest integer such that,
\begin{align*}
\text{AW}(R_1)+2\rhd y,\text{AW}(R_2)+2\rhd y,\ldots.
\end{align*}
Consider $G$.
\begin{align*}
\text{If } G~\|~\bigstar \text{, then AW}(G)=0. \\[-10pt]
\text{If } G>\bigstar \text{, then AW}(G)=y. \\[-10pt]
\text{If } G<\bigstar \text{, then AW}(G)=x.
\end{align*}
\end{theorem}

In most situations, the exception is invoked and it is necessary to compare $G$ to $\bigstar$. We will see this later in the analysis of the Mutual Failures Variant. There is one final useful result regarding the relation between a position's atomic weight and far star.

\begin{theorem}\label{t:BG:AtomicStar}
\doublespacing
For a position, $G$, $G+\bigstar>0$ exactly when AW$(G)\geq 1$.
\end{theorem}

\subsection{Graph Theory}

We will adopt the graph theory notation and definitions from \cite{GraphTheory}.   We will denote a graph $\graph{G} := (V, E)$ such that $V$ is a set of vertices, and $E$ is a set of edges.  We will assume throughout that $\graph{G}$ is a simple undirected graph.  For any vertex $v$, let $\delta(v)$ be the degree of $v$ in $G$.  A vertex with degree zero is called an isolated vertex and a vertex of degree one is called a leaf.  We will let $\path{n}$ denote a path graph on $n$ vertices and $\complete{n}$ be a complete graph on $n$ vertices. 

Our games will require players to make deletions from a graph thus many of the possible game positions are disconnected graphs. These disconnected positions can be considered a game sum of connected positions. As such, we will abuse notation and refer to the positions of a game sum as the components of the game sum for they are in fact the components of the overall graph position. Further, it is desirable to have a clear but related notation to distinguish between graphs and game positions. Thus, we will notate a graph class as $\graph{X}$ and the game positions played on that class as $X$.


\section{Results}

\subsection{The Mixed Deletion Game} \label{MixedGame}

\begin{definition}\label{d:Cl:BaseRules}
The Mixed Deletion Game (Classic Variant) is a combinatorial game defined by the following rules:
\begin{itemize}
\item The game is played on a finite graph.
\item There are two players, Left and Right, who alternate turns.
\item Left must delete one vertex and all incident edges from the graph.
\item Right must delete one edge.
\item The game ends when a deletion creates an isolated vertex.
\item The player whose deletion created an isolated vertex loses.
\end{itemize}
\end{definition}

Notice in path graphs there is an inherent bias towards Left.  Since any path position will only be broken into smaller path position components, the ability for only Left to be able to move on $P_3$ causes any $\path{3}$ component to act as a ``free move'' for Left in the game sum of path position components.  No such ``free move'' path exists for Right, denying him the ability to counter the existence of a $P_3$ component.  This implies there is no path position $P_n$, with a negative integer or dyadic rational value.  To more formalize this argument, consider the following proposition: 

\begin{prop}\label{p:Cl:LeftBias}

There is no graph that contains a move for Right and no move for Left in the Classic Variant.
\end{prop}

\begin{proof}
Let $\graph{G}$ be a graph. Assume there exists a move for Right. This implies the existence of an edge, $\{a,b\}$, such that $\delta(a),\delta(b)>1$. Thus, $a$ must be adjacent to some vertex $v'\neq b$. Consider the degree of $v'$. If $\delta(v')=1$, then Left can legally delete $v'$. If $\delta(v')>1$, then Left can legally delete $a$. Therefore, if Right has a legal move in a graph position, then Left must also have a legal move.

\end{proof}

We can apply the logic of Proposition \ref{p:Cl:LeftBias} to path graph positions and find that it is impossible to have a path position with a negative numeric value. 

\begin{corr}\label{c:Cl:NoNegInt}
There is no path position, $P_n$, with a negative integer or dyadic rational value.
\end{corr}

\begin{proof}

On $\path{{n\geq 3}}$, Left can delete a leaf vertex without creating an isolated vertex, so the only path position without a move for Left is $P_2$. However, Right also has no move in $P_2$. Thus, in $P_n$ if $\GL=\emptyset$, then $\GR=\emptyset$. Thus, it is impossible for $k\in\Z^+$ to exist such that, $P_n=-k=\game{}{-(k-1)}$. Therefore, there does not exist any path position with an integer value less than zero. 

Further, for $j>0$ and an odd $q\in\Z$, the game value $\frac{q}{2^j}=\game{\frac{q-1}{2^j}}{\frac{q+1}{2^j}}$. So any negative dyadic rational for which $|q|<2^j$ implies the existence of a position $-\frac{1}{2}=\game{-1}{0}$, but as shown above there is no path position valued $-1$. Notice in the case where $|q| \geq 2^j$, we can express $\frac{q}{2^j}$ as a mixed number requiring both a negative integer value and a dyadic rational between $-1$ and $0$, both of which have already been shown to be impossible values for a path position. Therefore, there does not exist any path position with a negative dyadic rational value.

\end{proof}

Corollary \ref{c:Cl:NoNegInt} implies Right's inability to counter Left's ``free moves.'' Because Left can move $\leftmove{P_3}{P_2}$ and Right cannot move on $P_3$, every $P_3$ component is worth a game value of $1$. If Left adopts a strategy of creating a $P_3$ component with every one of her moves and then moving on the $P_3$ components only when no more can be created, Right would need to create a component of equal and opposite game value to counter Left's strategy. Corollary \ref{c:Cl:NoNegInt} shows this is impossible. Hence, Left will be the winner of $P_n$ for $n$ sufficiently large.

\begin{theorem}\label{t:Paths}
For an integer $n \geq 5$, Left will win $P_n$ in the Classic Variant regardless of which player moves first.
\end{theorem}

\begin{proof}
Let $n$ be a positive integer, and consider the path position $P_n$ in the Classic Variant. Notice 

\begin{align*}
P_2&=\game{}{},\\ 
P_3&=\game{0}{},\\
P_4&=\game{1}{0},\text{ and}\\ 
P_5&=\game{P_4,0}{P_3}. 
\end{align*}

Since $P_3=1$, we know if Left moves first and $n\geq 6$, she can move to $1+P_{n-4}$. However by Corollary \ref{c:Cl:NoNegInt}, we know there is no $k\in \Z^+$ such that $1+P_k \leq 0$. Thus, regardless of which player wins the $P_{n-4}$ position, Left will always have the last legal move on $P_3$ and win. We have not shown that no path position with value greater than 1 exists, and if one did it would surely be more optimal for Left to win by a larger margin. Yet, we have shown Left is able to win by a potentially not optimal strategy.

If Right moves first, we know he will want to move $P_n$ to a game sum without creating a $P_3$ position or path position long enough for Left to create a $P_3$ on her following turn. This implies Right will move to $P_2+P_{n-2}=P_{n-2}$ or $P_5+P_{n-5}$. 

Let us consider the case where Right moves to $P_{n-2}$ first. From above, we know Left wins playing first on $P_{m\geq 3}$. Hence, Left will win if Right moves $\rightmove{P_n}{P_{n-2}}$ for $n \geq 5$. 

Now consider the case where Right moves to $P_5+P_{n-5}$. We can assume $3 < n-5 < 6$. So if Right moved to $P_5+P_4$, Left will move $\leftmove{P_5+P_4}{P_5+P_3}$, and if Right moved to $P_5+P_5$, Left will move $\leftmove{P_5+P_5}{P_5}$. Right's move on $P_5$ must create $P_3$. Thus, Right has no winning strategy if he moves $\rightmove{P_n}{P_5+P_{n-5}}$.L

Therefore, for $n\geq 5$, Left is guaranteed to win $P_n$ regardless which player moves first.

\end{proof}

Solutions for more graph classes, including cycles and wheels, can be extrapolated similarly if they can be reduced by a single move into a graph class that has already been solved.  Cycles follow immediately from Theorem \ref{t:Paths} because the first move will move to $P_n$ or $P_{n-1}$.  

\begin{theorem}\label{t:Cl:Cycles}
For an integer $n\geq 3$ and $n\neq 5$, Left will win the Classic Variant played on the cycle position, $C_n$, regardless of which player moves first.
\end{theorem}

For wheel graphs, Left will want to remove the center vertex first, independent of Right's move or who goes first.  This will result in $P_n$ after Left's first move.  Thus we can apply Theorem \ref{t:Paths} and we have the following: 

\begin{theorem}\label{t:Cl:Wheels}
For an integer $n \geq 3$ and $n\neq 5$, Left will win the wheel position $W_n$ regardless of which player moves first.
\end{theorem}

For a complete graph, $K_n$, Left also has the advantage by removing a vertex incident to the edge removed by Right.  Thus Right can always be forced to play on a complete subgraph $K_{n-1}$.  Eventually Right will be forced to move $\rightmove{K_3}{P_3}$, and Left wins $P_3$. Therefore, we have the following result for complete graphs: 

\begin{theorem}\label{t:Cl:Completes}
For a positive integer $n\geq 3$, Left will win the complete graph position $K_n$.
\end{theorem}

Other graph classes, including star graphs and complete bipartite graphs, can be handled similarly.  All of which show an advantage for Left. 

The analysis of the Classic Variant begins to show the imbalance between an edge deletion and a vertex deletion. Paths, cycles, wheels, stars, complete bipartite, and complete graphs are biased towards Left winning, and this bias largely resulted from Left's ability to delete leaf vertices.  It then becomes natural to ask, ``Who wins if Left cannot delete leaf vertices?''


\subsection{The Forbidden Leaf Variant} \label{ForbiddenGame}

 In the interest of trying to remove the bias towards Left, we will add a rule that denies her the ability to delete leaf vertices.  

\begin{definition}\label{d:FL:FLRules}
The Forbidden Leaf Variant adopts all of the base rules and adds to it the condition that Left may not delete leaf vertices.
\end{definition}

It should be immediately clear that $P_2=P_3=0$. However, notice Left can no longer move on $P_4$ but Right can. For path graphs, this creates an advantage for Right because any $\path{4}$ component acts as a ``free move'' for Right and there are no ``free moves'' for Left.  This implies there is no path position $P_n$, with a positive integer or dyadic rational value.  So we see our additional condition is too strong and rather than eliminating the bias, it has shifted it from Left to Right.  To more formalize this argument, consider the following proposition: 

\begin{prop}\label{p:FL:RightBias}
There is no graph that contains a move for Left and no move for Right in the Forbidden Leaf Variant.
\end{prop}

\begin{proof}
Let $\graph{G}$ be a graph, and $G$ be the corresponding game position. Assume $\GL\neq\emptyset$. This implies the existence of a vertex, $v$, such that $\delta (v) >1$, and for any vertex, $v'$, that is adjacent to $v$, $\delta(v')>1$. However, deleting the edge $\{v,v'\}$ is a legal move for Right. Thus, if Left has a legal move in $G$, then Right must also have a legal move in $G$. Therefore, there is no graph that contains a move for Left and does not contain a move for Right.
\end{proof}

Thus the advantage has now shifted to Right's ability to play on $\path{4}$ whereas Left can not.  

\begin{corr}\label{c:FL:NoPosInt}
There is no path position $P_n$, with a positive integer or dyadic rational value.
\end{corr}

\begin{proof}
We know from Proposition \ref{p:FL:RightBias} that if $\GR=\emptyset$ then $\GL=\emptyset$. Thus, it is impossible for $k\in\Z^+$ to exist such that $P_n=k=\game{(k-1)}{}$. Therefore, there does not exist any path position with an integer value greater than zero.

Further, for $j>0$ and an odd $q\in\Z$, the game value $\frac{q}{2^j}=\game{\frac{q-1}{2^j}}{\frac{q+1}{2^j}}$. So any positive dyadic rational for which $|q|<2^j$ implies the existence of a position $\frac{1}{2}=\game{0}{1}$, but as shown above there is no path position valued $1$. Notice in the case where $|q| \geq 2^j$, we can express $\frac{q}{2^j}$ as a mixed number requiring both a positive integer value and a dyadic rational between $0$ and $1$, both of which have already been shown to be impossible values for a path position. Therefore, there does not exist any path position with a positive dyadic rational value.

\end{proof}

We will now use this result to show Right wins for sufficiently large path positions.  

\begin{theorem}\label{t:FL:Paths}
For an integer $n\geq 8$ and not equal to $11$, Right is guaranteed to win $P_n$ in the Forbidden Leaf Variant regardless of which player moves first.
\end{theorem}

\begin{proof}
Let $n$ be a positive integer, and consider the path position $P_n$ in the Forbidden Leaf Variant. Notice
\begin{align*}
P_2&=\zerogame, \\
P_3&=\zerogame, \\
P_4&=\game{}{0} \text{, and}\\
P_5&=\ast .
\end{align*}

Consider $n\geq 6$. Since $P_4=-1$, we know Right can adopt a strategy of creating a $P_4$ component with each move and we know Left has no ability to counter this strategy. If Right moves first, then he will be able to create a $P_4$ component and win. 

Hence, we may restrict our examination to the case where Left moves first. Left will avoid creating any $P_4$ component or $P_k$ component where $k$ is sufficiently large to allow Right to create a $P_4$ component on his subsequent move. Left has two options: move $\leftmove{P_n}{P_{n-4}}$ or move $\leftmove{P_n}{P_5+P_{n-6}}$. For the first strategy, if $n\geq 8$, Right will be able to win. For the second strategy, Left can only win if $n=11$ as $P_5 +P_5=0$. Notice Left can only attempt the second strategy if $n\geq 8$. 
\end{proof}

We can again extrapolate our result on paths to solve cycles since any edge or vertex deletion in a cycle will result in a path.  Wheel graphs do not follow easily from paths because of the inability of Right to move a wheel to a cycle and Left's desire to avoid cycles.   

\begin{theorem}\label{t:FL:Cycles}
For an integer $n\geq 8$ such that $n\neq 11$, Right is guaranteed to win $C_n$ in the Forbidden Leaf Variant regardless of which player moves first.
\end{theorem}

\begin{proof}
Let $n$ be a positive integer, and $C_n$ be the cycle position on $n$ vertices in the Forbidden Leaf Variant. If Right moves first, he must move $\rightmove{C_n}{P_n}$. Left will lose moving first on $P_k$ if $k\geq 8$ and $k\neq 11$. If Left moves first, she must move $\leftmove{C_n}{P_{n-1}}$. Right will win moving first on $P_k$ if $k \geq 4$. Therefore, Right will win $C_n$ if $n\geq 8$ and $n\neq 11$ regardless of which player moves first.

\end{proof}

We conclude the Forbidden Leaf Variant by considering complete graphs.  As with the Classic Variant,  regardless of who goes first, Right can be forced to always move on a complete subgraph.  This can continue until Right must move on $K_4$ to which Left will counter and move the graph position to $P_3$.  Thus Left will win on $K_n$ if $n \geq 5$ regardless of which player moves first.  

\begin{theorem}\label{t:FL:Completes}
For an integer $n\geq 5$, Left is guaranteed to win $K_n$ in the Forbidden Leaf Variant regardless of which player moves first.
\end{theorem}

\begin{proof}
Let $n$ be a positive integer. Notice $K_1=K_2=0$ and $K_3=\ast$. If $n=4$ and Left moves first, she must move $\leftmove{K_4}{K_3}$ and Right wins moving first on $\ast$. If $n=4$ and Right moves first, then any edge he deletes results in the same graph position due to symmetry. The corresponding graph has two vertices of degree 2 and two vertices of degree 3. Left can win this position by deleting a vertex of degree 3 which moves the graph position to $P_3$. If $n\geq 5$, and Left moves first, any vertex she deletes moves $\leftmove{K_n}{K_{n-1}}$. Hence, it will suffice to show that Left wins moving second. Again on account of symmetry, any single edge deleted from $\complete{n}$ results in one distinct graph after excluding isomorphisms, so Right only has one possible first move. If Left counters by deleting one of the two vertices incident to the edge Right deleted, she moves to $K_{n-1}$. Eventually, Right will be forced to move first on $K_4$. From above, we know Left can win if this happens. Therefore, Left will win $K_n$ if $n\geq 5$ regardless of which player moves first.

\end{proof}

Left still has the ability to win complete graphs because complete graphs are connected enough to overcome the bias towards Right. In this way, they are different than the other graph classes because Left can prevent the graph from breaking down into paths until the paths are small enough to be unplayable. 

We saw early in the analysis of the Forbidden Leaf Variant that our additional rule was too strong. The bias given to Left in the Classic Variant was given to Right in the Forbidden Leaf Variant. Although we did see Left still wins complete graphs, it is clear that for paths, the bias shifted to Right because of his ability to play on $P_4$.


\subsection{The Mutual Failures Variant}\label{MutualGame}

In both of the prior variants, we found that the game was inherently biased towards one player or the other. Examining the Classic Variant and the Forbidden Leaf Variant for similarities revealed an interesting relationship in the failure states of the players. In the Classic Variant, the only path graph that did not contain a move for Left was $P_2$, while Right could not move on $P_2$, and, $P_3$. In the Forbidden Leaf Variant, Right could not move on $P_2$, and $P_3$, while Left could not move on $P_2$, $P_3$, and $P_4$. In both cases one players set of failure states was a proper subset of the other player's, thus we move to examining the \textit{Mutual Failures Variant.}

\begin{definition}\label{d:MF:MFRules}
The Mutual Failures Variant adopts all of the base rules, and adds to it the condition that in any graph position,  $\GR=\emptyset$ if and only if  $\GL=\emptyset$.
\end{definition}

This added condition dramatically changes the nature of the Mixed Deletion game and requires a different technique for analysis.  The condition, $\GR=\emptyset$ if and only if  $\GL=\emptyset$, implies that every position in the Mutual Failures Variant is all-small and infinitesimal. 

We will need to utilize atomic weights to solve this variant. Although we can calculate the atomic weight for small games based on the game tree, it is impractical to do so for large games.  For the benefit of the reader, we will show the calculation of the atomic weight $P_5$ and $P_6$ using two different methods. 

\begin{example}\label{ex:MF:BaseCases}
Consider $P_5$ in the Mutual Failures Variant. Notice $P_4=\game{0}{0}$, so $P_5=\game{P_4,0}{0}={\uparrow}{\ast}$.  By Theorem \ref{t:BG:UsefulEquiv}, to show $\farstareq{P_5}{\uparrow}$, we need to show that ${\downarrow}{\bigstar}<P_5-\uparrow<{\uparrow}{\bigstar}$. Thus, it will suffice to show ${\uparrow}{\ast}+\bigstar>0$ and ${\downarrow}{\ast}+\bigstar<0$. 

First, we can fix $\bigstar=\ast 2$ as any larger nim heap will be reduced to $\ast 2$ by optimal players. If Left moves first on ${\uparrow}{\ast}+\ast 2$, she will move $\leftmove{{\uparrow}{\ast}+\ast 2}{\uparrow}$ which she wins. If Right moves first on ${\uparrow}{\ast}+\ast 2$, he can move to $\uparrow$, $\ast 2$, or $\uparrow+\ast 2$. Left obviously wins the first two, and can win the third moving $\leftmove{\uparrow+\ast 2}{\uparrow}$. Thus, ${\uparrow}{\ast}+\bigstar>0$. Since ${\downarrow}{\ast}+\bigstar=-({\uparrow}{\ast}+\bigstar)$, we also know ${\downarrow}{\ast}+\bigstar<0$. Therefore, AW$(P_5)=1$.

Let us now consider $P_6 = \game{P_5,P_2+P_3}{P_2+P_4,P_3+P_3}.$ Let 
\begin{align*}
g=& \{\text{AW}(P_5)-2,\text{AW}(P_2+P_3)-2 \\
&|\text{AW}(P_2+P_4)+2,\text{AW}(P_3+P_3)+2 \}. 
\end{align*}
After simplifying, $g=\game{-1}{2}=0$ which is an integer and thus invokes the exception of Theorem \ref{t:BG:AtomicCalc}. Then $x=0$ is the least integer such that $-1\lhd x$ and $y=1$ is the greatest integer such that $2\rhd y$. We must now compare $P_6$ to $\bigstar$.

Consider $P_6+\bigstar$. Again we can fix $\bigstar=\ast 2$. If Right moves first, he can move to $\ast 2$ or $\ast +\ast 2$. Left will win both. Since Left wins moving second, we know $P_6+\bigstar\geq 0$, so $P_6>\bigstar$. Therefore, AW$(P_6)=y=1$. \\

To compute some additional base case atomic weights, we used CGSuite, a program used to analyze combinatorial games: 

\begin{table}[H]\label{table:BaseCases}
\centering
\begin{tabular}{|c|c|}
\hline
Path Position & Atomic Weight \\
\hline 
$P_2$  & 0 \\
$P_3$  & 0 \\
$P_4$  & 0 \\
$P_5$  & 1 \\
$P_6$  & 1 \\
$P_7$  & 1 \\
$P_8$  & 1 \\
$P_9$  & 2 \\
$P_{10}$ & 2 \\
$P_{11}$ & 2 \\
$P_{12}$ & 2 \\
\hline
\end{tabular}
\caption {Atomic Weight for $P_2$ through $P_{12}$. }
\end{table}
\end{example}

As path positions move to sums of smaller path positions, and the smallest paths have integer valued atomic weights the exception will be invoked to find the atomic weight of every path position. For this reason, we will need to know how a path position compares to far star. Additionally, as we show every path position on 5 or more vertices is greater than far star, we can use Theorem \ref{t:BG:AtomicStar} to show a lower bound of the atomic weights of these positions.

\begin{theorem}\label{t:MF:FarStarCompare}
\doublespacing
For an integer $n\geq 5$, the following hold:
\begin{itemize}
\item $P_n>\bigstar$, and
\item AW$(P_n)\geq 1$.
\end{itemize}
\end{theorem}

\begin{proof}
For an integer $n\geq 5$, let $P_n$ be the path position on $n$ vertices in the Mutual Failures Variant. We will prove by strong induction that $P_n>\bigstar$ and AW$(P_n)\geq 1$. Notice $P_n>\bigstar$ holds if Left wins $P_n+\bigstar$ regardless which player moves first, so that $P_n+\bigstar>0$.

Base Cases:  After we have shown a base case to be a Left win or a first player win, we will assume that Right will not move to that position. We will also assume Right will not move to $\bigstar$ as Left can move a single nim heap directly to $0$. Thus, we will omit these base case positions from Right's move set. Further, for $k\geq 6$, Left can move $\leftmove{P_{k+1}+\bigstar}{P_k+\bigstar}$, where $P_k+\bigstar$ is an earlier base case. For this reason we will not consider Left's move set for $P_n$ when $n\geq 6$.

Consider $P_4+\bigstar$. Both players will move $P_4+\bigstar$ to $\ast+\ast=0$. So $P_4+\bigstar$ is a first player win. 

Consider $P_5+\bigstar$. If Left moves first on $P_5+\bigstar$, she will move to $P_5+\ast$. Since $P_5+\ast=\uparrow$, she will win moving first. If Right moves first on $P_5+\bigstar$, his move set is $\GR=\{P_5,P_5+\ast\}$. Left can move the first position directly to $0$, and she moves $\leftmove{P_5+\ast}{P_4+\ast=0}$, so Left wins if Right moves first, and $P_5+\bigstar>0$.

Consider $P_6+\bigstar$. Right's move set is $\GR=\{P_6,\ast+\bigstar,P_6+\ast\}$. Left can move the first two positions directly to $0$, and the last position to $P_5+\ast=\uparrow$, so Left wins if Right moves first, and $P_6+\bigstar>0$.

Consider $P_7+\bigstar$. Right's move set is $\GR=\{P_7,P_7+\ast\}$. Left wins $P_7=\uparrow$, and she will move $\leftmove{P_7+\ast}{\ast+\ast=0}$, so Left wins if Right moves first, and $P_7+\bigstar>0$.

Consider $P_8+\bigstar$. Right's move set is $\GR=\{P_8,P_8+\ast,P_4+P_4+\bigstar\}$. Notice $P_4+P_4+\bigstar=\bigstar$, so Right will not make this move. Left moves $\leftmove{P_8}{P_7=\uparrow}$, and she will move $\leftmove{P_8+\ast}{P_5+\ast=\uparrow}$, so Left wins if Right moves first, and $P_8+\bigstar>0$.

Consider $P_9+\bigstar$. Right's move set is $\GR=\{P_9,P_9+\ast,P_5+P_4+\bigstar\}$. Left moves $\leftmove{P_9}{P_4+P_4=0}$, and she will move $\leftmove{P_9+\ast}{P_5+\ast=\uparrow}$. Lastly, she will move $P_5+P_4+\bigstar$ to $P_5+P_4=\uparrow$, so Left wins if Right moves first, and $P_9+\bigstar>0$.

Consider $P_{10}+\bigstar$. Right's move set is $\GR=\{P_{10},P_{10}+\ast,P_5+P_5+\bigstar,P_4+P_6+\bigstar\}$. Left will move $\leftmove{P_{10}}{P_4+P_5=\uparrow}$, and she will move $\leftmove{P_{10}+\ast}{P_{10}}$. Right can only move $P_{10}$ to $P_8$, $P_7$, $P_5+P_5=\Uparrow$ and $P_6+P_4$. We have shown already how Left wins the first three, and Left will win $P_6+P_4$ by moving to $P_5+P_4=\uparrow$. As for the other first moves Right can make, Left is able to move the smallest path component to $0$ leaving a base case shown to be a Left win. Thus, Left wins if Right moves first, and $P_{10}+\bigstar>0$.

We will now proceed with strong induction on $n>10$. Assume for all integers, $k$, such that $5\leq k \leq n$, that $P_k+\bigstar>0$.

Consider $P_{n+1}+\bigstar$. If Left moves first, she will move $\leftmove{P_{n+1}+\bigstar}{P_n+\bigstar}$. By the induction hypothesis, $P_n+\bigstar$ is a Left win regardless which player moves first.

If Right moves first, his move set is $\GR=\{P_{n+1},P_{n+1}+\ast,P_{n-1}+\bigstar,P_{n-2}+\bigstar,P_{n-3}+P_4+\bigstar,P_a+P_b+\bigstar\}$ for some integers $a,b\geq 5$ such that $a+b=n+1$. Left will move $\leftmove{P_{n+1}}{P_{n-5}+P_5}$, and because $n-5\geq 5$, we can apply Theorem \ref{t:BG:AtomicStar} to see AW$(P_{n-5}+P_5)\geq 2$. Then by Theorem \ref{t:BG:TwoAhead}, Left will win. Similarly, Left will move $\leftmove{P_{n+1}+\ast}{P_{n-5}+P_5+\ast}$, and AW$(P_{n-5}+P_5+\ast)\geq 2$. Application of the induction hypothesis shows that $P_{n-1}+\bigstar$ and $P_{n-2}+\bigstar$ are both Left wins. Left will move $\leftmove{P_{n-3}+P_4+\bigstar}{P_{n-3}+\bigstar}$ which is also a Left win by the induction hypothesis. Finally, because $a,b\geq 5$, $P_a+\bigstar$ and $P_b+\bigstar$ are both Left wins, so by Theorem \ref{t:BG:AtomicStar} AW$(P_a+P_b+\bigstar)\geq 2$. Thus, Left will win $P_{n+1}+\bigstar$ regardless which player moves first.

Therefore, $P_n>\bigstar$. Further, by Theorem \ref{t:BG:AtomicStar}, AW$(P_n)\geq 1$.

\end{proof}

Knowing that $P_n>\bigstar$, we can approach finding the atomic weight of $P_n$ inductively. We will use the fact that atomic weights are additive to determine the maximum atomic weight to which Left can move and the minimum atomic weight to which Right can move. The game position composed of Left's maximum minus two and Right's minimum plus two will be an integer requiring a comparison of $P_n$ to $\bigstar$. However, Theorem \ref{t:MF:FarStarCompare} makes this a trivial exercise. We then know AW$(P_n)$ is the greatest integer such that AW$(P_n)$ is less than two more than Right's minimum.

\begin{theorem}\label{t:MF:AWofPn}
\doublespacing
For an integer $n \geq 5$ and path position, $P_n$, the atomic weight of $P_n$ is $\left\lceil\frac{n}{4}\right\rceil-1$.
\end{theorem}

\begin{proof}
For an integer $n\geq 5$, let $P_n$ be the path position on $n$ vertices in the Mutual Failures Variant. We will prove by strong induction that AW$(P_n)=\left\lceil\frac{n}{4}\right\rceil-1$. 

Base Cases: See Example \ref{ex:MF:BaseCases}.

We will now proceed with strong induction on $n>10$. Assume for all integers, $k$, such that $5\leq k \leq n$, that AW$(P_k)=\left\lceil\frac{n}{4}\right\rceil-1$.

Consider AW$(P_{n+1})$. We will apply Theorem \ref{t:BG:AtomicCalc} to find  AW$(P_{n+1})$. We know Left is going to want to maximize the atomic weight, so she will not move to $P_{n-2}$, $P_{n-3}$, or $P_{n-4}+P_4$ as these moves potentially decrease the atomic weight of the larger component while creating a smaller component with atomic weight $0$. So then, Left will move to $P_n$ or $P_a+P_{n-a}$ where $a$ is an integer and $5\leq a \leq n-5$. By the induction hypothesis, AW$(P_n)=\left\lceil\frac{n}{4}\right\rceil-1$. For her other first move, 
\begin{align*}
\text{AW}(P_a+P_{n-a})&=\text{AW}(P_a)+\text{AW}(P_{n-a}) \\
&=\left\lceil\frac{a}{4}\right\rceil-1+\left\lceil\frac{n-a}{4}\right\rceil-1 \\
&\geq\left\lceil\frac{n}{4}\right\rceil-2.
\end{align*}
However, since $\left\lceil{x}\right\rceil + \left\lceil{y}\right\rceil \leq \left\lceil{x+y}\right\rceil +1$, $\text{AW}(P_a+P_{n-a}) \leq \left\lceil\frac{n}{4}\right\rceil-1$. So we will assume Left will always move $\leftmove{P_{n+1}}{P_n}$ as this consistently yields the largest atomic weight. 

Let us now consider Right's first moves. We know Right is going to want to minimize the atomic weight, so he will move to $P_{n-3}+P_4$ rather than $P_{n-1}$ or $P_{n-2}$. He could also move $\rightmove{P_{n+1}}{P_b+P_{n+1-b}}$ for an integer $b$ such that $5\leq b\leq n-4$. Notice if $n~$mod$~4\neq 0$, then 
$$\left\lceil\frac{n}{4}\right\rceil=\left\lceil\frac{n+1}{4}\right\rceil\text{ and }  \left\lceil\frac{n-3}{4}\right\rceil=\left\lceil\frac{n+1}{4}\right\rceil-1.$$

Additionally, if $n~$mod$~4=0$, then 
$$\left\lceil\frac{n}{4}\right\rceil=\left\lceil\frac{n-3}{4}\right\rceil\text{ and }
\left\lceil\frac{n}{4}\right\rceil=\left\lceil\frac{n+1}{4}\right\rceil-1. $$

Then it follows that
\begin{align*}
\text{AW}(P_{n-3}+P_4)&=\text{AW}(P_{n-3}) \\
&=\left\lceil\frac{n-3}{4}\right\rceil-1\\
&=\left\lceil\frac{n+1}{4}\right\rceil-2.
\end{align*}
For his other first move,
\begin{align*}
\text{AW}(P_b+P_{n+1-b})&=\text{AW}(P_b)+\text{AW}(P_{n+1-b}) \\
&=\left\lceil\frac{b}{4}\right\rceil-1+\left\lceil\frac{n+1-b}{4}\right\rceil-1 \\
&\geq\left\lceil\frac{n+1}{4}\right\rceil-2.
\end{align*}
Similarly to the atomic weight of Left's move from $P_{n+1}$ to $P_a+P_{n-a}$, this atomic weight is bounded above by $\left\lceil\frac{n+1}{4}\right\rceil-1$. We will assume Right always moves $\rightmove{P_{n+1}}{P_{n-3}+P_4}$ as this consistently yields the lowest atomic weight.

By applying Theorem \ref{t:BG:AtomicCalc}, we see
\begin{align*}
\text{AW}(P_{n+1})&=\game{\text{AW}(P_n)-2}{\text{AW}(P_{n-3}+P_4)+2} \\
&=\game{\left\lceil\frac{n}{4}\right\rceil-3}{\left\lceil\frac{n+1}{4}\right\rceil}.
\end{align*}
This is clearly a number and the simplest number that can be assigned to this game position is $\left\lceil\frac{n}{4}\right\rceil-2$ which is an integer. Since the exception is invoked, we must compare $P_{n+1}$ to $\bigstar$. By Theorem \ref{t:MF:FarStarCompare}, we know $P_j>\bigstar$ for $j\geq 5$. So, AW$(P_{n+1})$ is the greatest integer, $y$, such that $y\lhd\left\lceil\frac{n+1}{4}\right\rceil$. Thus, AW$(P_{n+1})=\left\lceil\frac{n+1}{4}\right\rceil-1$.

Therefore, for an integer $n \geq 5$, the atomic weight of $P_n$ is $\left\lceil\frac{n}{4}\right\rceil-1$.

\end{proof}

We now have the atomic weight for any $P_n$, and this is the only information needed to solve the Mutual Failures Variant on paths.

\begin{theorem}\label{t:MF:Paths}
\doublespacing
For an integer $n\geq 9$, Left will win the path position $P_n$ in the Mutual Failures Variant regardless of which player moves first.
\end{theorem}

\begin{proof}
Let $n\geq 9$ be a positive integer, and let $P_n$ be the path position on $n$ vertices in the Mutual Failures Variant. By Theorem \ref{t:BG:TwoAhead}, we know Left wins a game position, $G$, if AW$(G)\geq 2$, and by Theorem \ref{t:MF:AWofPn}, we know AW$(P_n)=\left\lceil\frac{n}{4}\right\rceil-1$. Thus, AW$(P_n)\geq 2$ when $n\geq 9$. Therefore, for $n\geq 9$, Left will win the path position $P_n$ regardless of which player moves first.

\end{proof}

As should be expected at this point, we will use the solution for paths to solve cycles.  Since any move on a cycle will produce a path, we see for an integer $n\geq 10$, Left will win the cycle position $C_n$ regardless of which player moves first.  Thus we have the following immediate result:  

\begin{theorem}\label{t:MF:Cycles}
For an integer $n\geq 10$, Left will win the cycle position $C_n$ in the Mutual Failures Variant regardless of which player moves first.
\end{theorem}

In wheel graphs, Left will want to delete the hub vertex, moving to $P_n$ or $C_n$, both of which are won by Left.  Therefore, for an integer $n\geq 10$, Left will win the wheel position $W_n$ regardless of which player moves first.  Thus we have the following immediate result:

\begin{theorem}\label{t:MF:Wheels}
For an integer $n\geq 10$, Left will win the wheel position $W_n$ in the Mutual Failures Variant regardless of which player moves first.
\end{theorem}

Finally, complete graphs work exactly the same in the Mutual Failures Variant as they do in the Forbidden Leaf Variant. Any complete graph position on more than four vertices will reduce to the complete graph position on four vertices in which Left will win moving second. 

\begin{theorem}\label{t:FL:Completes}
For an integer $n\geq 5$, Left is guaranteed to win $K_n$ in the Mutual Failures Variant regardless of which player moves first.
\end{theorem}

The results obtained in the Mutual Failures Variant show that Left is favored to win. We sought to remove the bias from the Mixed Deletion Game by forcing Left and Right to have identical sets of unplayable positions, but ultimately failed to do so. However, we can conclude that this bias does not result from one player being able to stockpile ``free moves.'' The only difference between players is what they are legally allowed to delete from a graph. This bias must then result from the differences between an edge and a vertex deletion.

\bibliographystyle{abbrv}

\bibliography{MyBibFile}

\begin{thebibliography}{1}

\bibitem{SubtractionGame}
R.~Adams, J.~Dixon, J.~Elder, J.~Peabody, O.~Vega, and K.~Willis.
\newblock Combinatorial analysis of a subtraction game on graphs.
\newblock {\em International Journal of Combinatorics}, Art. ID 1476359:8 pp,
  2016.

\bibitem{LessonsInPlay}
M.~H. Albert, R.~J. Nowakowski, and D.~Wolfe.
\newblock {\em Lessons in Play: An Introduction to Combinatorial Game Theory}.
\newblock A K Peters, 2007.

\bibitem{WinningWays}
E.~R. Berlekamp, J.~H. Conway, and R.~K. Guy.
\newblock {\em Winning Ways for Your Mathematical Plays}.
\newblock Academic Press, 1982.

\bibitem{Sprouts}
M.~Copper.
\newblock Graph theory and the game of sprouts.
\newblock {\em The American Mathematical Monthly}, 100(No. 5):478--482, May
  1993.

\bibitem{GameEdgeRemoval}
R.~P. Gallant, G.~Gunther, B.~L. Hartnell, and D.~F. Rall.
\newblock A game of edge removal on graphs.
\newblock {\em Journal of Combinatorial Mathematics and Combinatorial
  Computing}, 57:75--82, 2006.

\bibitem{EdgeParity}
P.~Harding and P.~Ottaway.
\newblock Edge deletion games with parity rules.
\newblock {\em Integers: Electronic Journal of Combinatorial Number Theory},
  14, 2014.

\bibitem{VertexParity}
R.~Nowkowski and P.~Ottaway.
\newblock Vertex deletion games with parity rules.
\newblock {\em Integers: Electronic Journal of Combinatorial Number Theory},
  5(2), 2005.

\bibitem{GraphTheory}
R.~J. Trudeau.
\newblock {\em Introduction to Graph Theory}.
\newblock Dover Publications, 1993.

\end{thebibliography}

\end{document}